\documentclass{article}
\usepackage[]{graphicx}
\usepackage[]{amsmath,amssymb}
\usepackage[]{amsthm,enumerate}
\usepackage{verbatim,bm}
\usepackage{a4wide}

\usepackage{color}
\usepackage{xcolor}
\usepackage{url}

\newtheorem{theorem}{Theorem}[section]
\newtheorem{lemma}[theorem]{Lemma}

\theoremstyle{definition}

\newtheorem{example}[theorem]{Example}

\theoremstyle{remark}
\newtheorem{remark}[theorem]{Remark}

\newcommand{\defn}[1]{{\em #1}}

\def\0{{\bm 0}}   

\title{Hadamard matrices related to the projective planes}
\author{
 Hadi Kharaghani\thanks{Department of Mathematics and Computer Science, University of Lethbridge,
Lethbridge, Alberta, T1K 3M4, Canada. \texttt{kharaghani@uleth.ca}}
\and
  Sho Suda\thanks{Department of Mathematics,  National Defense Academy of Japan, Yokosuka, Kanagawa 239-8686, Japan. \texttt{ssuda@nda.ac.jp}}
}
\date{\today}

\begin{document}
\maketitle

\begin{abstract}
Let $n$ be the order of a (quaternary) Hadamard matrix. It is shown that the existence of a projective plane of order $n$ is equivalent to the existence of a balancedly multi-splittable (quaternary) Hadamard matrix of order $n^2$.
\end{abstract}

\section{Introduction}
K. A. Bush \cite{bush} was the first to establish a link between projective planes of even order and specific Hadamard matrices (now called Bush-type) in 1971. H. J. Ryser \cite{ryser} found the same connection as an application of factors of design matrix in 1977. Eric Verheiden \cite{Eric} showed that the existence of only four MOLS of size ten
would lead to a symmetric Bush-type Hadamard matrix of order 100 in 1981. \par Franc C. Bussemaker, Willem Haemers and Ted Spence \cite{Willem-ted} used an exhaustive search and found no strongly regular graph with parameters (36,15,6,6) and chromatic number six or equivalently a symmetric Bush-type Hadamard matrix of order 36. Many Bush-type Hadamard matrices of order 100 are constructed, but none is known to be symmetric. The proof of the nonexistence of a symmetric Bush-type Hadamard matrix of order 100 would be exciting, however, there has been no attempt at showing it so far. The nonexistence of the projective plane of order ten was finally established by a long computational method by C. W. H. Lam et al. in \cite{p10-3,p10}. 
\par Balancedly splittable Hadamard matrices were introduced by the authors in 2018 in \cite{hadi-sho}, and it is widely expanded in a recent paper by Jedwab et al. in \cite{jedwab}. It is known \cite{bsod} that the existence of a Hadamard matrix of order $4n$ would lead to a balancedly splittable Hadamard matrix of order $64n^2$. There is no balancedly splittable Hadamard matrix of order $4n^2$, $n$ odd, see \cite{hadi-sho}. 
The case of Hadamard matrices of order $16n^2$ remains open, and no balancedly splittable Hadamard matrix of order $144$ is known. \par Concentrating on the order $144$, the authors were led to some exotic classes of balancedly splittable Hadamard matrices which we have dubbed \emph{balancedly multi-splittable Hadamard matrices}. There is a balancedly multi-splittable Hadamard matrix of order $4^m$ for every positive integer $m$, and it seems that these are probably the only Hadamard matrices with this property. \par It will be shown that the existence of a projective plane of order $4n$ is equivalent to the existence of a balancedly multi-splittable Hadamard matrix of order $16n^2$ 
{provided that $4n$ is the order of a Hadamard matrix}. \par A similar equivalence between the projective plane of order $2n$, $n$ odd, and balancedly multi-splittable quaternary Hadamard matrices will be presented too. \par The connection between projective planes and Hadamard matrices shown in \cite{bush,ryser,Eric} are all one sided results in which from a projective plane of even order symmetric Bush-type Hadamard matrices are constructed. The proof of the nonexistence of a symmetric Bush-type Hadamard matrix of order 100 would be exciting; however, there has yet to be an attempt at showing it. 
The fact that less than half of the assumed nine MOLS are sufficient for their construction as shown in \cite{Eric} makes one wonder if there are no symmetric Bush-type Hadamard matrices of order 100, even though not a single one is found yet.


\section{Preliminaries}
\subsection{Hadamard matrices}

An $n\times n$ matrix $H$ is a \defn{Hadamard matrix of order $n$} if its entries are $1,-1$ and it satisfies $HH^\top=I_n$, where $I_n$ denotes the identity matrix of order $n$.  
A Hadamard matrix $H$ of order $n$ is said to be \defn{balancedly splittable} if there is an $\ell\times n$ submatrix $H_1$ of $H$ such that inner products for any two distinct column vectors of $H_1$ take at most two values. 
More precisely, there exist integers $a,b$ and the adjacency matrix $A$ of a graph such that $H_1^\top H_1=\ell I_n +a A+b(J_n-A-I_n)$, where $J_n$ denotes the all-ones matrix of order $n$. 
In this case we say that $H$ is balancedly splittable with respect $H_1$. Only the special case of $b=-a$ will be used in this note.

The same concept can be extended to orthogonal designs \cite{bsod}.  
Here, we adopt the following definition for quaternary Hadamard matrices.  
An $n\times n$ matrix $H$ is a \defn{quaternary Hadamard matrix of order $n$} if its entries are $\pm1,\pm i$ and it satisfies $HH^*=nI_n$.  
A quaternary Hadamard matrix $H$ of order $n$ is said to be \defn{balancedly splittable} if there is an $\ell\times n$ submatrix $H_1$ of $H$ such that
 the off-diagonal entries of $H_1^* H_1$ are in the set 
$$\{\varepsilon \alpha,\varepsilon \alpha^*,\varepsilon \beta,\varepsilon \beta^* \mid \varepsilon\in\{\pm1,\pm i\} \},$$ 
where $\alpha,\beta$ are some complex numbers.   
In this paper, we restrict to the case $\alpha=\beta$ and we say that a quaternary Hadamard matrix $H$ of order $n$ is balancedly splittable if $H_1^* H_1=\ell I+\alpha S$ where $\alpha$ is some positive real number and $S$ is a $(0,\pm1,\pm i)$-matrix with zero diagonal entries and nonzero off-diagonal entries.

\subsection{Orthogonal arrays}


An \emph{orthogonal array} of strength $t$ and index $\lambda$ is an $N \times k$ matrix over the set $\{1,\dots,q\}$ such that in every $N \times t$ subarray, each $t$-tuple in $S^t$ appears $\lambda$ times. We denote this property as OA$_\lambda(N,k,q,t)$. Note that $N=\lambda q^t$ and $(N,k,q,t)$ is the parameter of the orthogonal array.

For $t=2e$, the following lower bound on $N$ was shown by Rao (see \cite[Theorem~2.1]{KPS}), namely, $N\geq \sum_{i=0}^e \binom{k}{i}(q-1)^i$. An orthogonal array with parameters $(N,k,q,2e)$ is said to be complete if the equality holds in above. 

When $t=2$ and $\lambda=1$, the complete orthogonal array has the parameters OA$_1(q^2,q+1,q,2)$, and it is known that its existence is equivalent to that of a projective plane of order $q$. For the construction the orthogonal version of a projective plane is used in the next section.

The following lemmas will be used later. 
\begin{lemma}\label{lem:dis}
Let $A$ be an $N\times k$ matrix over $\{1,\ldots,q\}$. Write $A=\sum_{i=1}^{q} i A_i$, where $A_i$ ($i\in\{1,\ldots,q\}$) are disjoint $q^2 \times (q+1)$ $(0,1)$-matrices. 
Let $D$ be the distance matrix, ie., $D$ is an $N\times N$ matrix whose rows and columns indexed by the rows of $A$ with $(i,j)$-entry defined by the Hamming distance between the $i$-th row and the $j$-th row of $A$. 
Then 
$\sum_{i=1}^q A_iA_i^\top=k J_N-D$ holds.   
\end{lemma}
\begin{proof}
See the proof of \cite[Lemma~2.5 (i)]{KPS}. 
\end{proof}

\begin{lemma}\label{lem:oa}
Assume that there exists an orthogonal array $A$ with parameters $(q^2,q+1,q,2)$. 
Write $A=\sum_{i=1}^{q} i A_i$, where $A_i$ ($i\in\{1,\ldots,q\}$) are disjoint $q^2 \times (q+1)$ $(0,1)$-matrices. 
Then  the matrices $A_i$ satisfy
\begin{enumerate}
\item $\sum_{i=1}^{q} A_iA_i^\top= J_{q^2}+q I_{q^2}$,
\item $\sum_{i,j=1,i\neq j}^{q}A_iA_j^\top=q(J_{q^2}-I_{q^2})$.
\item Consider the code $C$ obtained from the rows of $A$. Let $\{i_1,\ldots,i_s\}$ be any $s$-element subset of $\{1,\ldots,q+1\}$. The code $C'$ obtained from $C$ by restricting the coordinates on the set $\{i_1,\ldots,i_s\}$ have the Hamming distances $s$ or $s-1$ between the codewords in $C'$.  
\end{enumerate}
\end{lemma}
\begin{proof} 
The proof for (i) and (ii) are exactly the same as \cite[Lemma~2.5]{KPS}.   

The assumed orthogonal array is a $2$-design and $1$-distance set with Hamming distance $q$ in the Hamming association scheme. 
The case (iii) follows from the fact that $C$ is a $1$-distance set with Hamming distance $q$. 
\end{proof}

\begin{lemma}\label{lem:tight}{\rm \cite[Theorem~5.14]{D}}
Let $C$ be an equi-distance code of length $q+1$ over the symbol set $\{1,\ldots,q\}$. Then 
$$
|C|\leq q^2
$$
holds. Equality holds if and only if the matrix whose rows  consists of the codewords of $C$ is an orthogonal array OA$_{1}(q^2,q+1,q,2)$.  
\end{lemma}

\section{Balancedly multi-splittable Hadamard martrices}
We consider the following property on a Hadamard matrix. 
Let $H$ be a Hadamard matrix of order $4n^2$. 
Assume that $H$ is normalized so that the first column of $H$ is the all-ones vector. 
A Hadamard matrix $H$ is said to be \defn{balancedly multi-splittable} if there is a block form of $H=\begin{bmatrix}
{\bm 1} & H_1 & \cdots & H_{2n+1}  
\end{bmatrix}$ 
such that $H$ is balancedly spllitable with respect to a submatrix $\begin{bmatrix}
 H_{i_1} & \cdots & H_{i_n}  
\end{bmatrix}$ for any $n$-element subset $\{i_1,\ldots,i_n\}$ of $\{1,2,\ldots,2n+1\}$. 


The main results of this paper are as follows: 
\begin{theorem}\label{thm:main}
Let $n$ be a positive integer. The following are equivalent. 
\begin{enumerate}
\item There exists a balancedly multi-splittable Hadamard matrix of order $16n^2$. 
\item There exist an OA$_1(16n^2,4n+1,4n,2)$ and a Hadamard matrix of order $4n$.  
\end{enumerate}
\end{theorem}
\begin{theorem}\label{thm:mainq}
Let $n$ be a positive integer. The following are equivalent. 
\begin{enumerate}
\item There exists a balancedly multi-splittable quaternary Hadamard matrix of order $4n^2$
\item There exist an OA$_1(4n^2,2n+1,2n,2)$ and a quaternary Hadamard matrix of order $2n$. 
\end{enumerate}
\end{theorem}

\subsection{Proof of Theorem~\ref{thm:main}}
\begin{proof}[The proof of (i) $\Rightarrow$ (ii)]
Assume that there exists a Hadamard matrix  $H$ of order {$4n$}. 
Write $H$ as 
$$
H=\begin{bmatrix}
1 & r_1\\
1 & r_2\\
\vdots & \vdots\\
1 & r_{4n}\\
\end{bmatrix}, 
$$
where $r_i$ is a $1\times {(4n-1)}$ matrix for any $i$. 
\begin{lemma}\label{lem:hm}
\begin{enumerate}
\item For any $i$, $r_ir_i^\top={4n-1}$. 
\item For any distinct $i,j$, $r_i r_j^\top=-1$. 
\end{enumerate}

\end{lemma}

Assume that there exists an OA$(16n^2,4n+1,4n,2)$, say $A$, of index $1$ 
 over $\{1,\ldots,4n\}$. 
 Write $A=\sum_{i=1}^{4n} iA_i$, where the $A_i$'s are disjoint {$16n^2 \times (4n+1)$} $(0,1)$-matrices.
We then define the {$16n^2 \times (16n^2-1)$} matrix $D$ by $D=\sum_{i=1}^{4n} A_i\otimes r_i$ and $\tilde{D}=\begin{bmatrix}{\bm 1} &D \end{bmatrix}$.
\begin{lemma}
\begin{enumerate}
\item $DD^\top={16n^2 I_{16n^2}-J_{16n^2}}$. 
\item $\tilde{D}$ is a Hadamard matrix of order $16n^2$. 
\end{enumerate}
\end{lemma}
\begin{proof}
(i): By Lemma~\ref{lem:oa} and Lemma~\ref{lem:hm},  
\begin{align*}
DD^\top&=\sum_{i,j=1}^{4n} A_i A_j^\top \otimes r_i r_j^\top\\\displaybreak[0]
&=\sum_{i=1}^{4n} A_i A_i^\top \otimes r_i r_i^\top+\sum_{i\neq j} A_i A_j^\top \otimes r_i r_j^\top\\\displaybreak[0]
&=({4n-1})\sum_{i=1}^{4n} A_i A_i^\top -\sum_{i\neq j} A_i A_j^\top \\\displaybreak[0]
&=({4n-1}) J_{16n^2}+({4n-1})\cdot 4n I_{16n^2}- 4n(J_{16n^2}-I_{16n^2})\\\displaybreak[0]
&=16n^2 I_{16n^2}-J_{16n^2}. 
\end{align*}
(ii) immediately follows from (i). 
\end{proof}

Let $A'$ be a submatrix of $A$ obtained by restricting the columns to a $2n$ element set. 
Write $A'=\sum_{i=1}^{4n} i A'_i$, where $A'_i$ ($i\in\{1,\ldots,4n\}$) are disjoint $16n^2 \times 2n$ $(0,1)$-matrices.
\begin{lemma}\label{lem:A'}
There exists a symmetric $(0,1)$-matrix $B$ with diagonal entries $0$ such that 
\begin{enumerate}
\item $\sum_{i=1}^{4n} {A'}_i{A'}_i^\top=2n J_{16n^2}-(2nB+({2n-1})(J_{16n^2}-I_{16n^2}-B))$, and
\item $\sum_{i,j=1,i\neq j}^{4n}{A'}_i{A'}_j^\top=2nB+({2n-1})(J_{16n^2}-I_{16n^2}-B)$.
\end{enumerate}
\end{lemma}
\begin{proof}
Since the distance matrix of the code of rows of $A'$ is $2nB+({2n-1})(J_{16n^2}-I_{16n^2}-B)$ with the desired property, the case (i) follows from Lemma~\ref{lem:dis}. 

Since $\sum_{i=1}^{4n} {A'}_i=J_{16n^2,2n}$, we have $\sum_{i,j=1}^{4n} {A'}_i {A'}_j^\top=(\sum_{i=1}^{4n} {A'}_i)(\sum_{j=1}^{4n} {A'}_j^\top)=J_{16n^2,2n}J_{2n,16n^2}=2n J_{16n^2}$. This with (i) shows  (ii). 
\end{proof}
Now we consider $D'=\sum_{i=1}^{4n} {A'}_i\otimes r_i$. 
Then, by Lemma~\ref{lem:A'}, 
\begin{align*}
D'{D'}^\top&=\sum_{i,j=1}^{4n} {A'}_i {A'}_j^\top \otimes r_i r_j^\top\\\displaybreak[0]
&=\sum_{i=1}^{4n} {A'}_i {A'}_i^\top \otimes r_i r_i^\top+\sum_{i\neq j} {A'}_i {A'}_j^\top \otimes r_i r_j^\top\\\displaybreak[0]
&=({4n-1})\sum_{i=1}^{4n} {A'}_i {A'}_i^\top -\sum_{i\neq j} {A'}_i {A'}_j^\top \\\displaybreak[0]
&=({4n-1}) (2n J_{16n^2}-(2nB+({2n-1})(J_{16n^2}-I_{16n^2}-B))- (2nB+({2n-1})(J_{16n^2}-I_{16n^2}-B))\\\displaybreak[0]
&= ({8n^2-2n})I_{16n^2}+2n (J_{16n^2}-I_{16n^2}-2B). 
\end{align*}
Therefore the Hadamard matrix $D$ is balancedly multi-splittable.  
\end{proof}

\begin{proof}[The proof of (ii) $\Rightarrow$ (i)]
Assume that $H$ is a balancedly multi-splittable Hadamard matrix of order $16n^2$ with respect to the following block form: 
$$
H=\begin{bmatrix}
{\bm 1} & H_1 & \cdots & H_{{4n+1}}  
\end{bmatrix},
$$ 
where each $H_i$ is a $16n^2\times {(4n-1)}$ matrix. 

\begin{lemma}\label{lem:11-1}
For any $i$, $H_i H_i^\top$ is a $({4n-1},-1)$-matrix.   
\end{lemma}
\begin{proof}
We show the case $i=1$. 
Since $H$ is a Hadamard matrix of order $16n^2$, $HH^\top=16n^2I_{16n^2}$, that is, 
$$
J_{16n^2}+\sum_{i=1}^{{4n+1}}H_i H_i^\top=16n^2I_{16n^2}. 
$$ 
By the assumption of balanced multi-splittabllity, we have that the inner product of distinct rows of matrices $\begin{bmatrix} H_2 & \cdots & H_{{2n+1}}  
\end{bmatrix}$ or $\begin{bmatrix}
 H_{{2n+2}} & \cdots & H_{{4n+1}}  
\end{bmatrix}$ are $\pm 2n$. Thus, 
$$
\sum_{i=2}^{{2n+1}}H_i H_i^\top={(8n^2-2n)}I_{16n^2}+2nS,\quad \sum_{i={2n+2}}^{{4n+1}}H_i H_i^\top={(8n^2-2n)}I_{16n^2}+2nS', 
$$
where $S$ and $S'$ are $(0,1,-1)$-matrices with diagonal entries $0$ and off-diagonal entries $\pm1$. 
Then 
\begin{align*}
H_1H_1^\top&=16n^2I_{16n^2}-J_{16n^2}-({(16n^2-4n)}I_{16n^2}+2nS+2nS')\\
&=4nI_{16n^2}-J_{16n^2}-2n(S+S').
\end{align*}
Since both $S$ and $S'$ are $(0,\pm1)$-matrix, $S+S'$ is a $(0,\pm2)$-matrix with diagonal entries $0$.  However, the off-diagonal entries of $H_1H_1^\top$ cannot be ${-4n-1}$, $S+S'$ is $(0,-2)$-matrix. 
Therefore, $H_1H_1^\top$ is a $({4n-1},-1)$-matrix. 
\end{proof}

For each $i$,  consider the matrix $\tilde{H}_i=\begin{bmatrix}
{\bm 1} & H_i  \end{bmatrix}$. 
Then, by Lemma~\ref{lem:11-1}, $\tilde{H}_i\tilde{H}_i^\top$ is a $(4n,0)$-matrix. 
Since $\tilde{H}_i^\top \tilde{H}_i=16n^2I_{4n}$, the rank of $\tilde{H}_i$ is $4n$. 
Therefore there exist $4n$ rows of $\tilde{H}_i$ that correspond to the rows of
a Hadamard matrix $\tilde{K}_i$ of order $4n$.
 
Write $\tilde{K}_i=\begin{bmatrix}
{\bm 1} & K_i  \end{bmatrix}$.  
Assign a symbol $j$ to any row in $H_i$, which equals the $j$-th row of $K_i$. Let $A$ be the resulting $16n^2\times {(4n+1)}$ matrix over the symbol set $\{1,\ldots,4n\}$. 

\begin{lemma}
The code $C$ with codewords consisting of the rows of $A$ is an equidistance code with the number of codewords $16n^2$, equidistance $4n$, of length ${4n+1}$. 
\end{lemma}
\begin{proof}
It is enough to see the case for the first row and second row. 
Let the first and second rows of $H$ be the following forms:
\begin{align*}
\begin{bmatrix}
1 & r_{1,1} & \cdots & r_{1,{4n+1}}  
\end{bmatrix}, 
\\
\begin{bmatrix}
1 & r_{2,1} & \cdots & r_{2,{4n+1}}  
\end{bmatrix}. 
\end{align*}
Consider the inner product between them: 
$$
1+\sum_{i=1}^{{4n+1}}r_{1,i}r_{2,i}^\top=0.
$$
By Lemma~\ref{lem:11-1}, $r_{1,i}r_{2,i}^\top\in\{{4n-1},-1\}$ for any $i$. Then there exists $i_0$ such that 
$r_{1,i_0}r_{2,i_0}^\top={4n-1}$ and $r_{1,i}r_{2,i}^\top=-1$ for any $i\neq i_0$. 
Therefore the distance between the first row and second row is $4n$. 
\end{proof}

Since the code $C$ attains the upper bound in Lemma~\ref{lem:tight}, $A$ is an orthogonal array OA$_1(16n^2,{4n+1},4n,2)$.  
\end{proof}

\subsection{Proof of Theorem~\ref{thm:mainq}}
\begin{proof}[The proof of (i) $\Rightarrow$ (ii)]
Assume that there exists a quaternary Hadamard matrix $H$ of order $2n$. 
Write $H$ as 
$$
H=\begin{bmatrix}
1 & r_1\\
1 & r_2\\
\vdots & \vdots\\
1 & r_{2n}\\
\end{bmatrix}, 
$$
where $r_i$ is a $1\times {(2n-1)}$ matrix for any $i$. 
\begin{lemma}\label{lem:qhm}
\begin{enumerate}
\item For any $i$, $r_ir_i^*={2n-1}$. 
\item For any distinct $i,j$, $r_i r_j^*=-1$. 
\end{enumerate}

\end{lemma}

Assume that there exists an OA$(4n^2,{2n+1},2n,2)$, say $A$, of index $1$ 
 over $\{1,\ldots,2n\}$. Write $A=\sum_{i=1}^{2n} iA_i$, where the $A_i$'s are disjoint $4n^2\times {(2n+1)}$ $(0,1)$-matrices. 
We then define the $4n^2 \times {(4n^2-1)}$ matrix $D$ by $D=\sum_{i=1}^{2n} A_i\otimes r_i$ and $\tilde{D}=\begin{bmatrix}{\bm 1} &D \end{bmatrix}$.
\begin{lemma}
\begin{enumerate}
\item $DD^\top=4n^2 I_{4n^2}-J_{4n^2}$. 
\item $\tilde{D}$ is a quaternary Hadamard matrix of order $4n^2$. 
\end{enumerate}
\end{lemma}
\begin{proof}
(i): By Lemma~\ref{lem:oa} and Lemma~\ref{lem:qhm},  
\begin{align*}
DD^\top&=\sum_{i,j=1}^{2n} A_i A_j^\top \otimes r_i r_j^\top\\\displaybreak[0]
&=\sum_{i=1}^{2n} A_i A_i^\top \otimes r_i r_i^\top+\sum_{i\neq j} A_i A_j^\top \otimes r_i r_j^\top\\\displaybreak[0]
&={(2n-1)}\sum_{i=1}^{2n} A_i A_i^\top -\sum_{i\neq j} A_i A_j^\top \\\displaybreak[0]
&={(2n-1)} J_{4n^2}+{(2n-1)}\cdot 2n I_{4n^2}- 2n(J_{4n^2}-I_{4n^2})\\\displaybreak[0]
&=4n^2 I_{4n^2}-J_{4n^2}. 
\end{align*}
(ii) immediately follows from (i). 
\end{proof}

Let $A'$ be a submatrix of $A$ obtained by restricting the columns to {an} $n$ element set. 
Write $A'=\sum_{i=1}^{2n} i A'_i$, where $A'_i$ ($i\in\{1,\ldots,2n\}$) are disjoint $4n^2 \times n$ $(0,1)$-matrices.
\begin{lemma}\label{lem:A'c}
There exists a symmetric $(0,1)$-matrix $B$ with diagonal entries $0$ such that 
\begin{enumerate}
\item $\sum_{i=1}^{2n} {A'}_i{A'}_i^\top=n J_{4n^2}-(nB+{(n-1)}(J_{4n^2}-I_{4n^2}-B))$, and 
\item $\sum_{i,j=1,i\neq j}^{2n}{A'}_i{A'}_j^\top=nB+{(n-1)}(J_{4n^2}-I_{4n^2}-B)$.
\end{enumerate}
\end{lemma}
\begin{proof}
Since the distance matrix of the code of rows of $A'$ is $nB+{(n-1)}(J_{4n^2}-I_{4n^2}-B)$ with the desired property, the case (i) follows from Lemma~\ref{lem:dis}. 

Since $\sum_{i=1}^{2n} {A'}_i=J_{4n^2,n}$, we have $\sum_{i,j=1}^{2n} {A'}_i {A'}_j^\top=(\sum_{i=1}^{2n} {A'}_i)(\sum_{j=1}^{2n} {A'}_j^\top)=J_{4n^2,n}J_{n,4n^2}=n J_{4n^2}$. This with (i) shows  (ii). 
\end{proof}
Now we consider $D'=\sum_{i=1}^{2n} {A'}_i\otimes r_i$. 
Then, by Lemma~\ref{lem:A'c},  
\begin{align*}
D'{D'}^\top&=\sum_{i,j=1}^{2n} {A'}_i {A'}_j^\top \otimes r_i r_j^\top\\\displaybreak[0]
&=\sum_{i=1}^{2n} {A'}_i {A'}_i^\top \otimes r_i r_i^\top+\sum_{i\neq j} {A'}_i {A'}_j^\top \otimes r_i r_j^\top\\\displaybreak[0]
&={(2n-1)}\sum_{i=1}^{2n} {A'}_i {A'}_i^\top -\sum_{i\neq j} {A'}_i {A'}_j^\top \\\displaybreak[0]
&={(2n-1)} (n J_{4n^2}-(nB+{(n-1)}(J_{4n^2}-I_{4n^2}-B))- (nB+{(n-1)}(J_{4n^2}-I_{4n^2}-B))\\\displaybreak[0]
&= {(2n^2-n)} I_{4n^2}+n (J_{4n^2}-I_{4n^2}-2B). 
\end{align*}
Therefore the quaternary Hadamard matrix $D$ is balancedly multi-splittable.  
\end{proof}

\begin{proof}[The proof of (ii) $\Rightarrow$ (i)]

Assume that $H$ is a balancedly multi-splittable quaternary Hadamard matrix of order $4n^2$ with respect to the following block form: 
$$
H=\begin{bmatrix}
{\bm 1} & H_1 & \cdots & H_{{2n+1}}  
\end{bmatrix},
$$ 
where each $H_i$ is a $4n^2\times {(2n-1)}$ matrix. 

\begin{lemma}\label{lem:9-1}
For any $i$, $H_i H_i^*$ is a $({2n-1},-1)$-matrix.   
\end{lemma}
\begin{proof}
We show the case $i=1$. 
Since $H$ is a quaternary Hadamard matrix of order $4n^2$, $HH^*=4n^2I_{4n^2}$, that is, 
$$
J_{4n^2}+\sum_{i=1}^{{2n+1}}H_i H_i^*=4n^2I_{4n^2}. 
$$ 
By the assumption of balanced multi-splittabllity, we have that the inner product of distinct rows of matrices $\begin{bmatrix} H_2 & \cdots & H_{{n+1}}  
\end{bmatrix}$ or $\begin{bmatrix}
 H_{{n+2}} & \cdots & H_{{2n+1}}  
\end{bmatrix}$ are $\pm 2n,\pm2i$. Thus, 
$$
\sum_{i=2}^{{n+1}}H_i H_i^*={(2n^2-n)}I_{4n^2}+nS,\quad \sum_{i={n+2}}^{{2n+1}}H_i H_i^*={(2n^2-n)}I_{4n^2}+nS', 
$$
where $S$ and $S'$ are $(0,\pm1,\pm i)$-matrix with diagonal entries $0$ and off-diagonal entries $\pm1,\pm i$. 
Then 
\begin{align*}
H_1H_1^*&=4n^2I_{4n^2}-J_{4n^2}-({(4n^2-2n)}I_{100}+nS+nS')\\
&=2nI_{4n^2}-J_{4n^2}-n(S+S').
\end{align*}
Since both $S$ and $S'$ are $(0,\pm1,\pm i)$-matrix, $S+S'$ is a $(0,\pm2,\pm2 i)$-matrix with diagonal entries $0$.  However, the absolute values of off-diagonal entries of $H_1H_1^*$ cannot exceed ${2n-1}$, $S+S'$ is $(0,-2)$-matrix. 
Therefore, $H_1H_1^*$ is a $({2n-1},-1)$-matrix. 
\end{proof}

For each $i$,  consider the matrix $\tilde{H}_i=\begin{bmatrix}
{\bm 1} & H_i  \end{bmatrix}$. 
Then, by Lemma~\ref{lem:9-1}, $\tilde{H}_i\tilde{H}_i^*$ is a $(2n,0)$-matrix. 
Since $\tilde{H}_i^* \tilde{H}_i=4n^2I_{2n}$, the rank of $\tilde{H}_i$ is $2n$. 

Therefore there exist $4n$ rows of $\tilde{H}_i$ that correspond to the rows of
a Hadamard matrix $\tilde{K}_i$ of order $4n$.

Therefore there exist $2n$ rows of $\tilde{H}_i$ that correspond to the rows of
a Hadamard matrix $\tilde{K}_i$ of order $2n$.



Write $\tilde{K}_i=\begin{bmatrix}
{\bm 1} & K_i  \end{bmatrix}$.  
Assign a symbol $j$ to any row in $H_i$, which equals the $j$-th row of $K_i$. Let $A$ be the resulting $4n^2\times {(2n+1)}$ matrix over the symbol set $\{1,\ldots,2n\}$.

\begin{lemma}
The code $C$ with codewords consisting of the rows of $A$ is an equidistance code with the number of codewords $4n^2$, equidistance $2n$, of length ${2n+1}$. 
\end{lemma}
\begin{proof}
It is enough to see the case for the first row and second row. 
Let the first and second rows of $H$ be the following forms:
\begin{align*}
\begin{bmatrix}
1 & r_{1,1} & \cdots & r_{1,{2n+1}}  
\end{bmatrix}, 
\\
\begin{bmatrix}
1 & r_{2,1} & \cdots & r_{2,{2n+1}}  
\end{bmatrix}. 
\end{align*}
Consider the inner product between them: 
$$
1+\sum_{i=1}^{{2n+1}}r_{1,i}r_{2,i}^*=0.
$$
By Lemma~\ref{lem:9-1}, $r_{1,i}r_{2,i}^*\in\{{2n-1},-1\}$ for any $i$. Then there exists $i_0$ such that 
$r_{1,i_0}r_{2,i_0}^*={2n-1}$ and $r_{1,i}r_{2,i}^*=-1$ for any $i\neq i_0$. 
Therefore the distance between the first row and second row is $2n$. 
\end{proof}

Since the code $C$ attains the upper bound in Lemma~\ref{lem:tight}, $A$ is an orthogonal array OA$_1(4n^2,{2n+1},2n,2)$.  
\end{proof}

\section{Example}
In this section, we present an example of balancedly multi-splittable Hadamard matrices following the construction in Theorem~\ref{thm:main}.

\begin{example}
Take an OA$_1(16,5,4,2)$ $A$ and a Hadamard matrix $H$ of order $4$ as: 
\begin{align*}
A^\top&=\sum_{i=1}^4 A_i^\top=\left[
\begin{array}{cccccccccccccccc}
 1 & 1 & 1 & 1 & 2 & 2 & 2 & 2 & 3 & 3 & 3 & 3 & 4 & 4 & 4 & 4 \\
 1 & 2 & 3 & 4 & 1 & 2 & 3 & 4 & 1 & 2 & 3 & 4 & 1 & 2 & 3 & 4 \\
 1 & 2 & 3 & 4 & 2 & 1 & 4 & 3 & 3 & 4 & 1 & 2 & 4 & 3 & 2 & 1 \\
 1 & 2 & 3 & 4 & 3 & 4 & 1 & 2 & 4 & 3 & 2 & 1 & 2 & 1 & 4 & 3 \\
 1 & 2 & 3 & 4 & 4 & 3 & 2 & 1 & 2 & 1 & 4 & 3 & 3 & 4 & 1 & 2 \\
\end{array}
\right],\\
H&=\begin{bmatrix}
1 & r_1\\
1 & r_2\\
1 & r_3\\
1 & r_4\\
\end{bmatrix}=\left[
\begin{array}{cccc}
 1 & 1 & 1 & 1 \\
 1 & -1 & 1 & -1 \\
 1 & 1 & -1 & -1 \\
 1 & -1 & -1 & 1 \\
\end{array}
\right].
\end{align*}

Then the matrix $D$ constructed in Theorem~\ref{thm:main} is a balancedly multi-splittable Hadamard amtrix of order $16$: 
\begin{align*}
D&=\sum_{i=1}^4 A_i\otimes r_i\\
&=\begin{bmatrix} {\bm 1} & H_1 & H_2 & H_3 & H_4 & H_5 \end{bmatrix} \\
&=\left[
\begin{array}{c|ccc|ccc|ccc|ccc|ccc}
 1 & 1 & 1 & 1 & 1 & 1 & 1 & 1 & 1 & 1 & 1 & 1 & 1 & 1 & 1 & 1 \\
 1 & 1 & 1 & 1 & -1 & 1 & -1 & -1 & 1 & -1 & -1 & 1 & -1 & -1 & 1 & -1 \\
 1 & 1 & 1 & 1 & 1 & -1 & -1 & 1 & -1 & -1 & 1 & -1 & -1 & 1 & -1 & -1 \\
 1 & 1 & 1 & 1 & -1 & -1 & 1 & -1 & -1 & 1 & -1 & -1 & 1 & -1 & -1 & 1 \\
 1 & -1 & 1 & -1 & 1 & 1 & 1 & -1 & 1 & -1 & 1 & -1 & -1 & -1 & -1 & 1 \\
 1 & -1 & 1 & -1 & -1 & 1 & -1 & 1 & 1 & 1 & -1 & -1 & 1 & 1 & -1 & -1 \\
 1 & -1 & 1 & -1 & 1 & -1 & -1 & -1 & -1 & 1 & 1 & 1 & 1 & -1 & 1 & -1 \\
 1 & -1 & 1 & -1 & -1 & -1 & 1 & 1 & -1 & -1 & -1 & 1 & -1 & 1 & 1 & 1 \\
 1 & 1 & -1 & -1 & 1 & 1 & 1 & 1 & -1 & -1 & -1 & -1 & 1 & -1 & 1 & -1 \\
 1 & 1 & -1 & -1 & -1 & 1 & -1 & -1 & -1 & 1 & 1 & -1 & -1 & 1 & 1 & 1 \\
 1 & 1 & -1 & -1 & 1 & -1 & -1 & 1 & 1 & 1 & -1 & 1 & -1 & -1 & -1 & 1 \\
 1 & 1 & -1 & -1 & -1 & -1 & 1 & -1 & 1 & -1 & 1 & 1 & 1 & 1 & -1 & -1 \\
 1 & -1 & -1 & 1 & 1 & 1 & 1 & -1 & -1 & 1 & -1 & 1 & -1 & 1 & -1 & -1 \\
 1 & -1 & -1 & 1 & -1 & 1 & -1 & 1 & -1 & -1 & 1 & 1 & 1 & -1 & -1 & 1 \\
 1 & -1 & -1 & 1 & 1 & -1 & -1 & -1 & 1 & -1 & -1 & -1 & 1 & 1 & 1 & 1 \\
 1 & -1 & -1 & 1 & -1 & -1 & 1 & 1 & 1 & 1 & 1 & -1 & -1 & -1 & 1 & -1 \\
\end{array}
\right]. 
\end{align*}
\end{example}
\begin{remark}
There exist no balancedly multi-splittable quaternary Hadamard matrices of orders 36 and 100.
\end{remark}

\section*{Acknowledgments.}
Useful conversations with Professor Tayfeh-Rezaie is appreciated. Hadi Kharaghani is supported by the Natural Sciences and
Engineering  Research Council of Canada (NSERC).  Sho Suda is supported by JSPS KAKENHI Grant Number 18K03395, 22K03410.

\end{document}